\documentclass[review]{elsarticle}
\usepackage{amsmath,amssymb,amsfonts,amsthm,fancyhdr}
\usepackage{lineno,hyperref}
\modulolinenumbers[5]

\newtheorem{theorem}{Theorem}[section]
\newtheorem{lemma}[theorem]{Lemma}
\newtheorem{proposition}[theorem]{Proposition}

\theoremstyle{definition}
\newtheorem{definition}[theorem]{Definition}

\newtheorem{remark}[theorem]{Remark}
\newtheorem{corollary}[theorem]{Corollary}

\begin{document}
\begin{frontmatter}
\title{On the exponent of convergence of Engel series}

\author{Lei Shang}
\ead{auleishang@gmail.com}
\address{School of Mathematics, South China University of Technology, Guangzhou 510640, P.R. China}

\author{Min Wu}
\ead{wumin@scut.edu.cn}
\address{School of Mathematics, South China University of Technology, Guangzhou 510640, P.R. China}

\begin{abstract}\par
For $x\in (0,1)$, let $\langle d_1(x),d_2(x),d_3(x),\cdots \rangle$ be the Engel series expansion of $x$. Denote by $\lambda(x)$ the exponent of convergence of the sequence $\{d_n(x)\}$, namely
\begin{equation*}
\lambda(x)= \inf\left\{s \geq 0: \sum_{n \geq 1} d^{-s}_n(x)<\infty\right\}.
\end{equation*}
It follows from Erd\H{o}s, R\'{e}nyi and Sz\"{u}sz (1958) that $\lambda(x) =0$ for Lebesgue almost all $x\in (0,1)$. This paper is concerned with the topological and fractal properties of the level set $\{x\in (0,1): \lambda(x) =\alpha\}$ for $\alpha \in [0,\infty]$.
For the topological properties, it is proved that each level set is uncountable and dense in $(0,1)$. Furthermore, the level set is of the first Baire category for $\alpha\in [0,\infty)$ but residual for $\alpha =\infty$.
For the fractal properties, we prove that the Hausdorff dimension of the level set is as follows:
\[
\dim_{\rm H} \big\{x \in (0,1): \lambda(x) =\alpha\big\}=\dim_{\rm H} \big\{x \in (0,1): \lambda(x) \geq\alpha\big\}=
\left\{
  \begin{array}{ll}
    1-\alpha, & \hbox{$0\leq \alpha\leq1$;} \\
    0, & \hbox{$1<\alpha \leq \infty$.}
  \end{array}
\right.
\]
\end{abstract}


\end{frontmatter}

\section{Introduction}
Let $\{a_n\}$ be a finite or infinite sequence of positive integers. The \emph{exponent of convergence} of $\{a_n\}$, denoted by $\lambda:=\lambda(\{a_n\})$,
is defined as the infimum of $s \geq 0$ for which the series $ \sum_{n \geq 1} a^{-s}_n$ converges. That is,
\[
\lambda= \inf\left\{s \geq 0: \sum_{n \geq 1} a^{-s}_n<\infty\right\}.
\]
Note that the exponent of convergence of a finite sequence is equal to zero. For an infinite sequence $\{a_n\}$, if it is non-decreasing, then $\lambda$ can be described in terms of the growth rate of $a_n$:
\begin{equation}\label{lam}
\lambda = \limsup_{n \to \infty} \frac{\log n}{\log a_n},
\end{equation}
see \cite[p.\,26]{PS}.
The exponent of convergence has been extensively studied in several fields of mathematics, such as complex analysis and fractal geometry, see \cite{Bea, Fal90, FS, KS82, PS, Salat}.
In particular, \u{S}al\'{a}t \cite{Salat} investigated some topological and fractal properties for the exponent of convergence of a certain sequence associated with the dyadic representation of real numbers.

In this paper, we are concerned with the exponent of convergence of the digits in Engel series.
Let $T:[0,1)\to [0,1)$ be the \emph{Engel series map} defined as $T(0):=0$ and
\[
T(x) = x\left\lceil\frac{1}{x}\right\rceil -1,\ \ \ \forall x \in (0,1),
\]
where $\lceil y\rceil$ denotes the least integer not less than $y$. For $x\in (0,1)$, put $d_1(x) = \lceil 1/x\rceil$ and $d_{n+1}(x) = d_1(T^n(x))$ for $n \geq 1$.
Then $x$ admits a series of the form
\begin{equation}\label{EE}
x= \frac{1}{d_1(x)} + \frac{1}{d_1(x)d_2(x)}+\frac{1}{d_1(x)d_2(x)d_3(x)}+\cdots,
\end{equation}
which is called the \emph{Engel series expansion} of $x$. Here $d_1(x), d_2(x),d_3(x), \cdots$ are positive integers and called the \emph{digits} of the Engel series expansion of $x$.
If there exists $k\in \mathbb{N}$ such that $T^k(x)=0$, then we say that the Engel series expansion of $x$ is \emph{finite} and write the right-hand side of (\ref{EE}) as $\langle d_1(x),d_2(x),\cdots,d_k(x)\rangle$;
otherwise the Engel series expansion of $x$ is said to be \emph{infinite} and the right-hand side of (\ref{EE}) will be denoted by $\langle d_1(x),d_2(x),\cdots,d_k(x),\cdots\rangle$.
It was shown in \cite[p.\,7]{ERS58} that $x$ is irrational if and only if its Engel series expansion is infinite, and $2 \leq d_1(x) \leq \cdots \leq d_{n-1}(x) \leq d_n(x)\leq \cdots$ with
$d_n(x) \to \infty$ as $n \to \infty$ for any irrational number $x \in (0,1)$. For example, the Euler number $e$ has a simple Engel series expansion $e-2=\langle 2,3,\cdots,n,\cdots\rangle$ which implies that $e$ is irrational and its digits tend to infinity with a linear growth speed. For the asymptotic behavior of $d_n$, a classical result of Erd\H{o}s, R\'{e}nyi and Sz\"{u}sz \cite[Theorem 3]{ERS58} says that for Lebesgue almost all $x\in (0,1)$,
\begin{equation}\label{slln}
\lim_{n \to \infty} \frac{\log d_n(x)}{n} = 1.
\end{equation}
This means that for Lebesgue almost all $x\in (0,1)$, the growth rate of $d_n(x)$ is exponential.
See Galambos \cite{lesGal76} for more results of Engel series.

Let $\lambda(x)$ be the exponent of convergence of the sequence of digits in the Engel series expansion of $x$. Namely
\begin{equation}\label{cal}
\lambda(x):= \inf\left\{s \geq 0: \sum_{n \geq 1} d^{-s}_n(x)<\infty\right\}.
\end{equation}
It is easy to see that $\lambda(x)$ takes values in $[0,\infty]$, $\lambda(x) =0$ for all rational numbers $x\in(0,1)$, and $\lambda(e-2)=1$.
In fact, we will see below that there are uncountably many irrational numbers such that their exponents of convergence can be any prescribed real number (see Theorem \ref{Dense}).
By \eqref{lam} and (\ref{slln}), $\lambda(x) =0$ for Lebesgue almost all $x\in (0,1)$. In other words, $\{x\in (0,1): \lambda(x) =0\}$ has full Lebesgue measure but the level set $\{x\in (0,1): \lambda(x) =\alpha\}$ is of Lebesgue measure zero for $\alpha \in (0,\infty]$. We would like to study the fine structure and size of these level sets.
From a topological point of view, we will prove that for $\alpha \in [0,\infty)$, the set $\{x \in (0,1): \lambda(x) \leq\alpha\}$ is of the first Baire category but the set $\{x \in (0,1): \lambda(x) =\infty\}$ is residual,
see Theorem \ref{RF}. Hence for $\alpha \in [0,\infty)$, the level set $\{x\in (0,1): \lambda(x) =\alpha\}$ is of the first Baire category and the set $\{x \in (0,1): \lambda(x) \geq\alpha\}$ is residual.
For the fractal properties, $\{x \in (0,1): \lambda(x) \leq\alpha\}$ has full Hausdorff dimension since $\lambda(x) =0$ for Lebesgue almost all $x\in (0,1)$.
However, we will show the following multifractal analysis result in Theorem \ref{lambda}:
\begin{equation*}
\dim_{\rm H} \big\{x \in (0,1): \lambda(x) =\alpha\big\}=\dim_{\rm H} \big\{x \in (0,1): \lambda(x) \geq\alpha\big\}=
\left\{
  \begin{array}{ll}
    1-\alpha, & \hbox{$0\leq \alpha\leq1$;} \\
    0, & \hbox{$1<\alpha\leq \infty$.}
  \end{array}
\right.
\end{equation*}
Here and in the sequel $\dim_{\rm H}$ denotes the Hausdorff dimension.
By \eqref{lam} and \eqref{cal}, $\lambda(x)$ can be written as the limsup of $\frac{\log n}{\log d_n(x)}$ for all irrational numbers $x\in (0,1)$.
Hence the multifractal analysis of $\lambda(x)$ is closely related to the Hausdorff dimension of certain sets associated with the growth rate of the digits.
For more information on this topic, we refer the reader to \cite{FWnon, Liu, LW03, SW, SW2} and references therein.

The paper is organized as follows. Section 2 is devoted to several definitions and useful lemmas of the Engel series.
In Section 3, we prove Theorems \ref{Dense} and \ref{RF}. The proof of Theorem \ref{lambda} is given in Section 4.

%

\section{Preliminaries}
We begin with several definitions and basic properties of Engel series.

\begin{definition}
A finite sequence $(\sigma_1, \sigma_2, \cdots, \sigma_n) \in \mathbb{N}^n$ is said to be \emph{admissible} for Engel series if there exists $x \in (0,1)$ such that $d_1(x) = \sigma_1, d_2(x) = \sigma_2, \cdots, d_n(x) = \sigma_n$. An infinite sequence $(\sigma_1, \sigma_2, \cdots, \sigma_k, \cdots) \in \mathbb{N}^\mathbb{N}$ is said to be \emph{admissible} for Engel series if there exists $x \in (0,1)$ such that $d_1(x) = \sigma_1, d_2(x) = \sigma_2, \cdots, d_k(x) = \sigma_k$ for all $k \geq 1$.

\end{definition}

Denote by $\Sigma_n$ the collection of all admissible sequences with length $n$ and by $\Sigma$ that of all infinite admissible sequences.
The following result gives a characterisation of admissible sequences.

\begin{proposition}[\cite{lesGal76}]\label{AD}
A finite sequence $(\sigma_1, \sigma_2, \cdots, \sigma_n) \in \Sigma_n$ if and only if $2 \leq \sigma_1 \leq \sigma_2 \leq \cdots \leq \sigma_n$.
An infinite sequence $(\sigma_1, \sigma_2, \cdots, \sigma_k, \cdots) \in \Sigma$ if and only if
\[
\sigma_1 \geq 2, \ \ \ \ \ \ \ \sigma_{k+1} \geq \sigma_k, \ \forall\ k \geq 1 \ \ \ \ \ \text{and} \ \ \ \ \ \lim_{k \to \infty}\sigma_k = \infty.
\]
\end{proposition}

\begin{definition}
Let $(\sigma_1, \sigma_2, \cdots, \sigma_n) \in \Sigma_n$. We call
\begin{equation*}
I_n(\sigma_1, \sigma_2, \cdots, \sigma_n) := \big\{x \in (0,1): d_1(x)=\sigma_1,d_2(x)=\sigma_2,\cdots,d_n(x)=\sigma_n\big\}
\end{equation*}
the cylinder of order $n$ of the Engel series.
\end{definition}

Now we give some basic facts on the structure and length of cylinders of the Engel series. We use $|I|$ to denote the diameter of a subset $I$ of $(0,1)$.

\begin{proposition}[{\cite[p. 84]{lesGal76}}]\label{cylinder}
Let $(\sigma_1, \sigma_2, \cdots, \sigma_n) \in \Sigma_n$. Then the cylinder $I_n(\sigma_1, \sigma_2, \cdots, \sigma_n)$ is an interval with the left endpoint
\[
\frac{1}{\sigma_1}+\cdots+\frac{1}{\sigma_1\sigma_2\cdots \sigma_{n-1}}+ \frac{1}{\sigma_1\sigma_2\cdots \sigma_{n-1}\sigma_n}
\]
and the right endpoint
\[
\frac{1}{\sigma_1}+\cdots+ \frac{1}{\sigma_1\sigma_2\cdots \sigma_{n-1}}+ \frac{1}{\sigma_1\sigma_2\cdots \sigma_{n-1}(\sigma_n-1)}.
\]
Moreover,
\begin{equation*}\label{cylinder length}
\left|I_n(\sigma_1, \sigma_2, \cdots, \sigma_n)\right| = \frac{1}{\sigma_1\sigma_2\cdots \sigma_{n-1}\sigma_n(\sigma_n-1)}.
\end{equation*}
\end{proposition}

The following result, due to \cite[Theorem 4.4 \& Lemma 4.5]{FWnon}, is often used to give the upper bound for the Hausdorff dimension of certain sets arising in Engel series.

\begin{lemma}[\cite{FWnon}]\label{2}
Let $\varphi:\mathbb{N} \rightarrow\mathbb{R}^{+}$ be a function and $N\geq 1$ be an integer. Then
\[
\dim_{\rm H}\big\{x \in (0,1): d_n(x) \geq \varphi(n), \forall n \geq N\big\} = \frac{1}{B},
\]
where $B$ is given by
\[
\log B := \limsup\limits_{n \to \infty} \frac{\log\log\varphi(n)}{n}.
\]
\end{lemma}

The following lemma provides a powerful method to estimate the lower bound for the Hausdorff dimension of fractal sets associated with the growth rate of the digits in Engel series.

\begin{lemma}[{\cite[Lemma 2.6.]{SW2}}]\label{1}
Let $\{t_{n}\}_{n \geq 1}$ be a non-decreasing sequence of real numbers with $t_{1} \geq2$ and $t_{n}\to \infty$ as $n \to \infty$. Write
\[
\mathbb{E}(\{t_n\}):=\big\{x\in(0,1): nt_{n}<d_{n}(x)\leq(n+1)t_{n}, \forall n\geq1\big\}.
\]
Then
\[
\dim_\mathrm{H}\mathbb{E}(\{t_{n}\})=\frac{1}{1+\eta},
\]
where $\eta$ is given by
\[
\eta:= \limsup_{n \to \infty} \frac{\log(n+1)!+\log t_{n+1}}{\log (t_{1}\cdots t_{n})}.
\]
\end{lemma}

\section{Topological properties}
In this section, we will study the topological properties of $\lambda: (0,1)\to [0,\infty]$. More precisely, we are interested in the fine structure of the level set $$\{x \in (0,1): \lambda(x) =\alpha\}$$ for any $0 \leq \alpha\leq \infty$. A natural question that arises here is whether these level sets are empty? That is to say, given $0 \leq \alpha\leq \infty$, does there exist $x_0 \in (0,1)$ such that $\lambda(x_0) = \alpha$?
We will give a positive answer to this question by showing the intermediate value property of $\lambda$.
See Oxtoby \cite{Oxt} for the relevant terminologies in topology.

\begin{lemma}\label{IVP}
For any $0 \leq \alpha\leq \infty$, there exists an irrational number $x \in (0,1)$ such that $\lambda(x) =\alpha$.
\end{lemma}
\begin{proof}
For $\alpha=0$, the proof is trivial. For $\alpha \in (0,\infty)$, let $x:=\langle\sigma_1,\sigma_2,\cdots,\sigma_n,\cdots\rangle$ with $\sigma_n =\lceil(n+1)^{1/\alpha}\rceil$ for all $n \geq 1$.
Then we have $2 \leq \sigma_1 \leq \sigma_2 \leq \cdots \leq \sigma_n \leq \cdots$ and $\sigma_n \to \infty$ as $n\to \infty$.
By Proposition \ref{AD}, we deduce that $x$ is irrational and $d_n(x)=\sigma_n$ for all $n \geq 1$. In view of \eqref{lam} and \eqref{cal}, we obtain
\[
\lambda(x)= \limsup_{n \to \infty} \frac{\log n}{\log d_n(x)} =\alpha.
\]
For $\alpha =\infty$, let $x:=\langle\sigma_1,\sigma_2,\cdots,\sigma_n,\cdots\rangle$ with $\sigma_1=\sigma_2=2$ and $\sigma_n =\lceil \log n\rceil$ for all $n \geq 3$. Then $x$ is irrational and $\lambda(x)=\infty$.
Therefore, the proof is completed.
\end{proof}

Furthermore, the following result shows that there are uncountably many irrational numbers such that their exponents of convergence can be any prescribed number.

\begin{theorem}\label{Dense}
For any $0 \leq \alpha\leq \infty$, the level set $\{x \in (0,1): \lambda(x) =\alpha\}$ is uncountable and dense in $(0,1)$.
\end{theorem}

\begin{proof}
For the case $\alpha =0$, it suffices to note that $\lambda(x) =0$ for Lebesgue almost all $x\in (0,1)$.
For $0 <\alpha\leq \infty$, it follows from Lemma \ref{IVP} that there exists an irrational number $x_0 \in (0,1)$ such that $\lambda(x_0) =\alpha$.
Note that $d_n(x_0) \to \infty$ as $n \to \infty$, so we can choose an infinite sequence $\{n_k\}_{k\geq 1}$ such that $d_{n_k}(x_0)< d_{n_k+1}(x_0)$ for all $k \geq 1$.
Let $(\varepsilon_1,\varepsilon_2,\cdots,\varepsilon_k,\cdots) \in \{0,1\}^{\mathbb{N}}$ be fixed. We define a new real number as
$\widehat{x}:= \langle\widehat{d}_1,\widehat{d}_2,\cdots,\widehat{d}_n,\cdots\rangle$
with
\[
\widehat{d}_n =
\left\{
  \begin{array}{ll}
    d_n(x_0), & \hbox{if $n\neq n_k$, $\forall\,k \geq 1$;} \\
    d_n(x_0)+\varepsilon_k, & \hbox{if $n=n_k$ for some $k \geq 1$.}
  \end{array}
\right.
\]
Then the sequence $\{\widehat{d}_n\}_{n\geq1}$ is non-decreasing with $\widehat{d}_1 \geq 2$ and $\widehat{d}_n \to \infty$ as $n \to \infty$.
By Proposition \ref{AD} and the algorithm of the Engel series expansion of $\widehat{x}$, we deduce that $d_n(\widehat{x}) =\widehat{d}_n$ and $ d_n(x_0) \leq d_n(\widehat{x}) \leq  d_n(x_0) +1$ for all $n \geq1$. According to the formula (\ref{cal}), we get $\lambda(\widehat{x}) = \lambda(x_0) = \alpha$. Note that the space $\{0,1\}^{\mathbb{N}}$ is uncountable, so there are uncountably infinite many irrationals $x \in (0,1)$ such that $\lambda(x)=\alpha$.

Next we will prove that the level set $\{x \in (0,1): \lambda(x) =\alpha\}$ is dense in $(0,1)$. To do this, let
\[
D(x_0):=\bigcup_{N \geq 1}\big\{x\in (0,1): d_n(x)= d_n(x_0), \forall\,n\geq N\big\}.
\]
Then $D(x_0)$ is a subset of $\{x \in (0,1): \lambda(x) =\alpha\}$. So it is sufficient to show $D(x_0)$ is dense in $(0,1)$.
Let $y \in (0,1)$ be fixed.

(i) If $y$ is rational, then it can be written as $y=\langle d_1(y),\cdots,d_k(y)\rangle$ for some $k\in \mathbb{N}$. For $m \geq 1$, let
$\ell_m= \inf\{\ell\geq k+2: d_{\ell}(x_0) \geq d_k(y)+m\}$. Since $d_n(x_0) \to \infty$ as $n \to \infty$, we have $\ell_m$ is finite, $\ell_{m+1} \geq \ell_m$ for all $m \geq 1$ and $\ell_m \to \infty$ as $m \to \infty$. Hence
\[
2 \leq d_1(y)\leq \cdots\leq d_k(y)< d_k(y)+m \leq d_{\ell_m}(x_0) \leq d_{\ell_m+1}(x_0) \leq d_{\ell_m+2}(x_0) \leq \cdots.
\]
Put
\[
y_m:=\langle d_1(y),\cdots,d_k(y),d_k(y)+m, \underbrace{ d_{\ell_m}(x_0),\cdots,d_{\ell_m}(x_0)}_{\ell_m -k-1}, d_{\ell_m+1}(x_0),d_{\ell_m+2}(x_0),\cdots\rangle.
\]
Then $d_n(y_m)=d_{n}(x_0)$ for all $n \geq \ell_m$, which yields $y_m \in D(x_0)$ for all $m \geq 1$.
Note that
\begin{align*}
0<y_m-y& =\frac{1}{d_1(y) \cdots d_k(y)} \left(\frac{1}{d_k(y)+m} + \frac{1}{(d_k(y)+m)d_{\ell_m}(x_0)}+\cdots\right)\\
& \leq \frac{1}{d_1(y) \cdots d_k(y)} \left(\frac{1}{d_k(y)+m} + \frac{1}{(d_k(y)+m)^2} +\cdots\right)\\
&= \frac{1}{d_1(y) \cdots d_k(y)(d_k(y)+m-1)}.
\end{align*}
Thus $y_m \to y$ as $m \to \infty$.

(ii) If $y$ is irrational, then we write $y=\langle d_1(y),\cdots,d_n(y),\cdots\rangle$. For $m \geq 1$,
let $\ell_m= \inf\{\ell\geq m+1: d_{\ell}(x_0) \geq d_m(y)\}$ and
\[
y_m:=\langle d_1(y),\cdots,d_m(y), \underbrace{ d_{\ell_m}(x_0),\cdots,d_{\ell_m}(x_0)}_{\ell_m -m}, d_{\ell_m+1}(x_0),d_{\ell_m+2}(x_0),\cdots\rangle.
\]
Then we have $d_n(y_m)=d_{n}(x_0)$ for all $n \geq \ell_m$ and hence $y_m \in D(x_0)$ for all $m \geq 1$. Note that $y_m$ and $y$ are both in the cylinder $I_m(d_1(y),\cdots,d_m(y))$, so we conclude from
Proposition \ref{cylinder} that
\[
0 <|y_m-y| \leq |I_m(d_1(y),\cdots,d_m(y))| \leq \frac{1}{2^m},
\]
which gives $y_m \to y$ as $m \to \infty$.

In both cases, we can always find a sequence $\{y_m\}_{m\geq 1}$ of numbers in $D(x_0)$ such that $y_m \to y$ as $m \to \infty$, i.e., $D(x_0)$ is dense in $(0,1)$. Then the desired results follows.

\end{proof}

Observe that a set and its closure have the same box-counting dimension (see Proposition 3.4 of \cite{Fal90}). As a consequence of Theorem \ref{Dense}, the level set $\{x \in (0,1): \lambda(x) =\alpha\}$ has full box-counting dimension. Denote by $\dim_{\rm B}$ the box-counting dimension.

\begin{corollary}
For any $0 \leq \alpha\leq \infty$, $\dim_{\rm B}\{x \in (0,1): \lambda(x) =\alpha\}=1$.
\end{corollary}

The following result is an immediate result of Theorem \ref{Dense} which shows that $\lambda(x)$ is discontinuous as a function of $x$ in $(0,1)$.

\begin{corollary}
$\lambda: (0,1)\to [0,\infty]$ is everywhere discontinuous.
\end{corollary}

\begin{proof}
Let $z \in (0,1)$ be fixed and denote $\beta:=\lambda(z)$. For a number $\gamma \neq \beta$, we deduce from Theorem \ref{Dense} that the level set $\{x \in (0,1): \lambda(x) =\gamma\}$ is dense in $(0,1)$. For the given $z$, there exists $\{z_n\}_{n\geq 1}$ such that $\lambda(z_n)=\gamma$ for all $n\geq 1$ and $z_n \to z$ as $n \to \infty$. Since $\lim_{n \to \infty}\lambda(z_n) = \gamma \neq \beta = \lambda(z)$, we obtain $\lambda: (0,1)\to [0,\infty]$ is discontinuous at $z$. Hence it is everywhere discontinuous.
\end{proof}

Next we would like to investigate the topological structure of the level set $\{x\in (0,1): \lambda(x) =\alpha\}$ in the sense of Baire category.
It will be shown that the level set is of the first Baire category for $\alpha \in [0,\infty)$.

\begin{theorem}\label{RF}
For any $0 \leq \alpha <\infty$, the set $\{x \in (0,1): \lambda(x) \leq\alpha\}$ is of the first Baire category in $(0,1)$ and the set $\{x \in (0,1): \lambda(x) =\infty\}$ is residual in $(0,1)$.
\end{theorem}

\begin{remark}
Note that any subset of a set of the first category is of the first category and any superset of a residual set is residual. For $\alpha \in [0,\infty)$, we deduce from Theorem \ref{RF} that the level set $\{x \in (0,1): \lambda(x) =\alpha\}$ is of the first category and the set $\{x \in (0,1): \lambda(x) \geq\alpha\}$ is residual.
\end{remark}

Denote by $\dim_{\rm P}$ the packing dimension.
It is known that if $E$ is a subset of $(0,1)$ with $\dim_{\rm P} E<1$, then it is of the first Baire category, see \cite[p. 65]{Edgar}. Combining this with Theorem \ref{RF}, we deduce that the set $\{x \in (0,1): \lambda(x) \geq\alpha\}$ has full packing dimension for $\alpha \in [0,\infty]$.

\begin{corollary}
For any $0 \leq \alpha\leq \infty$, $\dim_{\rm P} \{x \in (0,1): \lambda(x) \geq \alpha\} =1$.
\end{corollary}

We point out that for any $\alpha \in [0,\infty]$, the set $\{x \in (0,1): \lambda(x) \leq\alpha\}$ has full Lebesgue measure (then it has full Hausdorff/Packing/Box-counting dimension) because $\lambda(x) =0$ for Lebesgue almost all $x\in (0,1)$.

To prove Theorems \ref{RF}, we first give a useful lemma.

\begin{lemma}\label{Gdelta}
For any $0<\alpha <\infty$, the set $\{x \in (0,1): \lambda(x) \geq\alpha\}$ is a $G_\delta$ set in $(0,1)$.
\end{lemma}

\begin{proof}
Let $\mathbb{I}:= (0,1)\backslash\mathbb{Q}$. Note that $\lambda(x)=0$ for all $x \in (0,1)\cap\mathbb{Q}$.
For $\alpha \in (0,\infty)$, by \eqref{lam} and \eqref{cal}, we obtain
\begin{align}\label{B}
\big\{x \in (0,1): \lambda(x) \geq\alpha\big\} &= \left\{x \in \mathbb{I}:  \limsup_{n \to \infty} \frac{\log n}{\log d_n(x)} \geq \alpha\right\}\nonumber\\
&= \bigcap^\infty_{k=k_0} \bigcap^\infty_{N=1} \bigcup^\infty_{n=N} B(k,N,n),
\end{align}
where $k_0 := \lfloor1/\alpha\rfloor+1$ and $$B(k,N,n) = \left\{x \in \mathbb{I}: \frac{\log n}{\log d_n(x)} \geq \alpha -1/k\right\} = \left\{x \in \mathbb{I}: d_n(x) \leq n^{1/(\alpha -1/k)}\right\}.$$
Each non-empty set $B(k,N,n)$ can be written as a union of a finite number of open sets in $\mathbb{I}$. More precisely,
\begin{align*}
B(k,N,n) =\left\{x \in \mathbb{I}: d_n(x) \leq n^{1/(\alpha -1/k)}\right\} = \bigcup_{(\sigma_1,\cdots,\sigma_n)\in \mathcal{C}_n} \mathbb{I}\cap I_n(\sigma_1,\cdots,\sigma_n),
\end{align*}
where $\mathcal{C}_n:= \{(\sigma_1,\cdots,\sigma_n)\in \mathbb{N}^n: 2 \leq \sigma_1 \leq \cdots \leq \sigma_n \leq n^{1/(\alpha -1/k)}\}$ is a finite set (see Lemma 3.5 in \cite{SW}).
Since $\mathbb{I}\cap I_n(\sigma_1,\cdots,\sigma_n)$ is open in $\mathbb{I}$, we see that $B(k,N,n)$ is also open in $\mathbb{I}$ and then $\big\{x \in (0,1): \lambda(x) \geq\alpha\big\}$ is a $G_\delta$ set in $\mathbb{I}$ by \eqref{B}. Note that $\mathbb{I}$ is a $G_\delta$ set in $(0,1)$ and the intersection of two $G_\delta$ sets is still a $G_\delta$ set, we have $\{x \in (0,1): \lambda(x) \geq\alpha\}$ is a $G_\delta$ set in $(0,1)$.
\end{proof}

We are now in a position to give the proof of Theorem \ref{RF}. To prove a set is residual in $(0,1)$, it is enough to show it contains a dense $G_\delta$ subset of $(0,1)$, see for example \cite[Theorem 9.2]{Oxt}.

\begin{proof}[Proof of Theorem \ref{RF}]
For $\alpha \in [0,\infty)$, it follows from Theorem \ref{Dense} and Lemma \ref{Gdelta} that $\{x \in (0,1): \lambda(x) \geq \alpha +1\}$ is a dense $G_\delta$ set in $(0,1)$.
Then we have $\{x \in (0,1): \lambda(x) >\alpha\}$ is residual in (0,1), i.e., the set $\{x \in (0,1): \lambda(x) \leq\alpha\}$ is of the first Baire category.

According to Theorem \ref{Dense} and Lemma \ref{Gdelta}, we deduce that
$\big\{x \in (0,1): \lambda(x) \geq N\big\}$ is residual in $(0,1)$ for all $N\geq 1$.
Since
\[
\big\{x \in (0,1): \lambda(x) =\infty\big\} = \bigcap^\infty_{N=1} \big\{x \in (0,1): \lambda(x) \geq N\big\}
\]
and the intersection of countably many residual sets is still residual (see \cite[Theorem 1.4]{Oxt}),
we have $\{x \in (0,1): \lambda(x) =\infty\}$ is residual in $(0,1)$.
\end{proof}

\section{Multifractal analysis}

In this section, we propose to do the multifractal analysis of $\lambda: (0,1)\to [0,\infty]$, i.e., the Hausdorff dimension of the level set $\{x \in (0,1): \lambda(x) =\alpha\}$.
Recall that the set $\{x \in (0,1): \lambda(x) \leq\alpha\}$ has full Hausdorff/packing/box-counting dimension, and the set $\{x \in (0,1): \lambda(x) \geq\alpha\}$ has full packing dimension
(then it has full box-counting dimension), which leads to calculate its Hausdorff dimension.
It turns out that this set and the level set have the same Hausdorff dimension.

\begin{theorem}\label{lambda}
For any $0\leq \alpha \leq \infty$,
\begin{equation*}
\dim_{\rm H} \big\{x \in (0,1): \lambda(x) =\alpha\big\}=\dim_{\rm H} \big\{x \in (0,1): \lambda(x) \geq\alpha\big\}=
\left\{
  \begin{array}{ll}
    1-\alpha, & \hbox{$0\leq \alpha\leq1$;} \\
    0, & \hbox{$1<\alpha \leq \infty$.}
  \end{array}
\right.
\end{equation*}
\end{theorem}

With the conventions $\frac{1}{\infty}=0$ and $\frac{\infty}{\infty}=1$, Theorem \ref{lambda} is equivalent to the following theorem.
Let $D(x) =  \liminf_{n\to\infty}\frac{\log d_n(x)}{\log n}$.

\begin{theorem}\label{Lambda}
For any $0\leq \alpha \leq \infty$,
\begin{equation*}
\dim_{\rm H} \big\{x \in (0,1): D(x) =\alpha\big\}=\dim_{\rm H} \big\{x \in (0,1): D(x) \leq\alpha\big\}=
\left\{
  \begin{array}{ll} \vspace{0.1cm}
    0, & \hbox{$0\leq \alpha<1$;} \\
    \dfrac{\alpha -1}{\alpha}, & \hbox{$1\leq\alpha \leq \infty$.}
  \end{array}
\right.
\end{equation*}
\end{theorem}

Remark that Shang and Wu \cite[Theorem 3.2]{SW} proved
\begin{equation}\label{SWTheorem3.2}
\dim_{\rm H}\left\{x\in(0,1): \lim\limits_{n\to\infty}\frac{\log d_n(x)}{\log n}=\alpha\right\}=
\left\{
  \begin{array}{ll} \vspace{0.1cm}
    0, & \hbox{$0\leq \alpha<1$;} \\
    \dfrac{\alpha -1}{\alpha}, & \hbox{$1\leq\alpha \leq \infty$.}
  \end{array}
\right.
\end{equation}
This gives the lower bound for the Hausdorff dimension of $\{x \in (0,1): D(x) =\alpha\}$. In fact, the lower bound can also be obtained by choosing a suitable sequence $\{t_n\}_{n \geq 1}$ in Lemma \ref{2}.
So it remains to calculate the upper bound for the Hausdorff dimension of $\{x \in (0,1): D(x) \leq\alpha\}$. To this end, the following lemma is needed.

\begin{lemma}
(i) For $0 \leq \beta<1$,
\begin{equation}\label{beta}
\dim_{\rm H}\big\{x \in (0,1): D(x) \leq\beta\big\} = 0.
\end{equation}
(ii) For $1 \leq p\leq q<\infty$,
\begin{equation}\label{pq}
\dim_{\rm H}\big\{x \in (0,1): p \leq D(x) \leq q\big\} \leq \frac{q-1}{p}.
\end{equation}
\end{lemma}

\begin{proof}
Recall a combinatorial result (see  \cite[Lemma 3.5]{SW}):
for two positive integers $M \geq 2$ and $n \geq 1$,
\begin{align}\label{CL}
N_n(M) := \#\big\{(d_1, \cdots,d_n) \in \mathbb{N}^n: 2 \leq d_1 \leq \cdots \leq d_n \leq M\big\}=\frac{(n+M-2)!}{n!\cdot(M-2)!}.
\end{align}

(i) Let $0 \leq \beta<1$. For any $0<\varepsilon<1-\beta$, if $D(x) \leq\beta$ for some $x\in (0,1)$, then $d_n(x) \leq n^{\beta +\varepsilon}$ holds for infinitely many $n \in \mathbb{N}$. Namely
\[
\big\{x \in (0,1): D(x) \leq\beta\big\} \subseteq \bigcap^\infty_{N=N_0}\bigcup^\infty_{n =N} \big\{x\in (0,1): d_n(x) \leq n^{\beta +\varepsilon} \big\},
\]
where $N_0:= \lceil 2^{1/(\beta+\varepsilon)} \rceil$. Note that
\[
\big\{x\in (0,1): d_n(x) \leq n^{\beta +\varepsilon} \big\} = \bigcup_{(\sigma_1, \cdots, \sigma_n)\in \mathcal{D}_n}I_n(\sigma_1, \cdots, \sigma_n),
\]
where $\mathcal{D}_n$ is given by $\mathcal{D}_n=\left\{(\sigma_1, \cdots, \sigma_n)\in\mathbb{N}^{n}: 2\leq \sigma_1\leq \cdots\leq \sigma_n\leq n^{\beta+\varepsilon}\right\}$. Hence
\[
\big\{x \in (0,1): D(x) \leq\beta\big\} \subseteq \bigcap^\infty_{N=N_0}\bigcup^\infty_{n =N}\bigcup_{(\sigma_1, \cdots, \sigma_n)\in \mathcal{D}_n}I_n(\sigma_1, \cdots, \sigma_n),
\]
which implies that for fixed $N$, the family $$\big\{I_n(\sigma_1, \cdots, \sigma_n): n \geq N, (\sigma_1, \cdots, \sigma_n)\in \mathcal{D}_n\big\}$$ is a cover of $\{x \in (0,1): D(x) \leq\beta\}$. Since $0<\varepsilon<1-\beta$, it follows from (\ref{CL}) that
\begin{align*}
\# \mathcal{D}_n &=  N_{n}\left(\lfloor n^{\beta+\varepsilon}\rfloor\right)\\
&\leq (n+1)\cdot(n+2)\cdots(n+\lfloor n^{\beta+\varepsilon}\rfloor-2)\\
&\leq \left(n+ n^{\beta+\varepsilon} \right)^{n^{\beta+\varepsilon}}\\
& \leq 2^{(1+\log_2 n)n^{\beta+\varepsilon}}.
\end{align*}
From Proposition \ref{cylinder}, we get $|I_n(\sigma_1, \cdots, \sigma_n)| \leq 2^{-n}$ for all $(\sigma_1, \cdots, \sigma_n)\in \mathcal{D}_n$.
Let $\mathcal{H}^{s}$ denote the $s$-dimensional Hausdorff measure.
Then
\begin{align*}
\mathcal{H}^{\varepsilon}(\{x \in (0,1): D(x) \leq\beta\})&\leq \liminf_{N\to\infty}\sum\limits_{n=N}^{\infty}\sum\limits_{(\sigma_1,\cdots,\sigma_n)\in\mathcal{D}_n}|I_n(\sigma_1,\cdots,\sigma_n)|^{\varepsilon}\\
 &\leq \liminf_{N\to\infty}\sum\limits_{n=N}^{\infty} \#\mathcal{D}_n  \cdot 2^{-\varepsilon n}\\
&\leq  \liminf_{N\to\infty}\sum\limits_{n=N}^{\infty} 2^{-\varepsilon n +(1+\log_2 n)n^{\beta+\varepsilon}} \\
&= 0,
\end{align*}
which gives $\dim_{\rm H} \{x \in (0,1): D(x) \leq\beta\}\leq \varepsilon$. Letting $\varepsilon \to 0^+$, we obtain the desired result.

(ii) Let $1 \leq p\leq q<\infty$. For any $0<\varepsilon<p$, if $p \leq D(x) \leq q$ for some $x\in (0,1)$, then $d_n(x) \leq n^{q +\varepsilon}$ holds for infinitely many $n \in \mathbb{N}$ and $d_m(x) \geq m^{p -\varepsilon}$ holds for sufficiently large $m$. Hence there exists $N \geq 1$ such that for any $n \geq N$, there exists $k \geq n$ such that $d_k(x) \leq k^{q +\varepsilon}$ and $d_j(x) \geq j^{p-\varepsilon}$ for all $N \leq j \leq k$. That is to say,
\[
\big\{x \in (0,1): p \leq D(x) \leq q\big\}\subseteq \bigcup_{N=N_0}B_{N}(\varepsilon),
\]
where $B_{N}(\varepsilon)$ is defined as
\begin{equation}\label{bhgx}
B_{N}(\varepsilon)=\bigcap_{n=N}\bigcup_{k=n}\big\{x\in(0,1): d_k(x)\leq k^{q+\varepsilon},\ d_{j}(x)\geq j^{p-\varepsilon},\ \forall\,N\leq j\leq k\big\}.
\end{equation}
By the monotonicity and countable stability of Hausdorff dimension (see \cite[p. 32]{Fal90}), we obtain
\begin{equation}\label{wsgx1}
\dim_{\rm H}\big\{x \in (0,1): p \leq D(x) \leq q\big\}\leq\sup_{N\geq N_0}\big\{\dim_{\rm H}B_{N}(\varepsilon)\big\}.
\end{equation}
From now on, let $N\geq N_0$ be fixed. We will deal with the Hausdorff dimension of $B_{N}(\varepsilon)$.
In view of (\ref{bhgx}), for any $n \geq N$, we see that
\begin{equation}\label{3}
 B_{N}(\varepsilon) \subseteq \bigcup_{k=n}\bigcup_{(\sigma_1, \cdots, \sigma_k) \in \widehat{\mathcal{D}}_k}I_k(\sigma_1, \cdots, \sigma_k),
\end{equation}
where $\widehat{\mathcal{D}}_k=\big\{(\sigma_1, \cdots, \sigma_k)\in\Sigma_{k}:\sigma_k \leq k^{q+\varepsilon}, \sigma_j\geq j^{p-\varepsilon}, \forall\,N\leq j\leq k\big\}$.
For any $(\sigma_1, \cdots, \sigma_k) \in \widehat{\mathcal{D}}_k$, it follows from Proposition \ref{cylinder} that
$$|I_k(\sigma_1, \cdots, \sigma_k)|<(N\cdots k)^{-(p-\varepsilon)}=\left(\frac{k!}{(N-1)!}\right)^{-(p-\varepsilon)}.$$
Since $q +\varepsilon>1$, by (\ref{CL}), we have
\begin{align*}
\# \widehat{\mathcal{D}}_k\leq N_{k}\left(\lfloor k^{q+\varepsilon}\rfloor\right)&= \frac{\left(\lfloor k^{q+\varepsilon}\rfloor -1\right)\cdot \lfloor k^{q+\varepsilon}\rfloor \cdots\left(\lfloor k^{q+\varepsilon}\rfloor+k-2\right)}{k!}\nonumber\\
&\leq \frac{k^{k(q+\varepsilon)}}{k!}\cdot\left(1+\frac{1}{k^{q+\varepsilon}}\right)\cdots\left(1+\frac{k-1}{k^{q+\varepsilon}}\right)\\
& \leq \frac{2^{k-1}\cdot k^{k(q+\varepsilon)}}{k!}.
\end{align*}
Recall that the Stirling formula: $\sqrt{2\pi}n^{n+\frac{1}{2}}e^{-n}\leq n!\leq en^{n+\frac{1}{2}}e^{-n}$ for all $n \geq 1$.
Then
\[
\# \widehat{\mathcal{D}}_k< 2^{k-1} \cdot e^{(q+\varepsilon)k} \cdot (k!)^{q+\varepsilon-1}.
\]
From \eqref{3}, we see that for any $n \geq N$, the family $$\big\{I_k(\sigma_1, \cdots, \sigma_k): k \geq n, (\sigma_1, \cdots, \sigma_k) \in \widehat{\mathcal{D}}_k\big\}$$ is a cover of $B_{N}(\varepsilon)$.
Let $s=(q+2\varepsilon-1)/(p-\varepsilon)$. We have
\begin{align*}
\mathcal{H}^{s}( B_{N}(\varepsilon))&\leq \liminf_{n\to\infty}\sum\limits_{k=n}^{\infty}\sum\limits_{(\sigma_1,\cdots,\sigma_k)\in\widehat{\mathcal{D}}_k}|I_k(\sigma_1,\sigma_2,\cdots,\sigma_k)|^{s}\\
&\leq \liminf_{n\to\infty}\big((N-1)!\big)^{(p-\varepsilon)s}\cdot \sum\limits_{k=n}^{\infty} (k!)^{-(p-\varepsilon)s} \cdot 2^{k-1} \cdot e^{(q+\varepsilon)k} \cdot (k!)^{q+\varepsilon-1} \\
&= \big((N-1)!\big)^{(p-\varepsilon)s}\cdot \liminf_{n\to\infty}\sum\limits_{k=n}^{\infty} 2^{k-1} \cdot \frac{e^{(q+\varepsilon)k}}{(k!)^{\varepsilon}}\\
& =0.
\end{align*}
Then $\dim_{\rm H}  B_{N}(\varepsilon) \leq s$. By \eqref{wsgx1},
\[
\dim_{\rm H}\big\{x \in (0,1): p \leq D(x) \leq q\big\}\leq\sup_{N\geq N_0}\big\{\dim_{\rm H}B_{N}(\varepsilon)\big\} \leq \frac{q+2\varepsilon-1}{p-\varepsilon}.
\]
Letting $\varepsilon \to 0^+$ yields $\dim_{\rm H}\{x \in (0,1): p \leq D(x) \leq q\} \leq (q-1)/p$.
\end{proof}

Now we are ready to give the proof of Theorem \ref{Lambda}.

\begin{proof}[Proof of Theorem \ref{Lambda}]
We only need to prove
\[
\dim_{\rm H} \big\{x \in (0,1): D(x) \leq\alpha\big\}\leq
\left\{
  \begin{array}{ll} \vspace{0.1cm}
    0, & \hbox{$0\leq \alpha\leq 1$;} \\
    \dfrac{\alpha -1}{\alpha}, & \hbox{$1<\alpha \leq \infty$.}
  \end{array}
\right.
\]
For the case $\alpha =\infty$, the proof is trivial.
For $0 \leq \alpha\leq1$, it suffices to show
\[
\dim_{\rm H}\big\{x \in (0,1): D(x) \leq1\big\} = 0.
\]
In fact, since
\[
\big\{x \in (0,1): D(x)<1\big\} = \bigcup^\infty_{k=1} \left\{x \in (0,1): D(x)\leq 1 - \frac{1}{k}\right\},
\]
by (\ref{beta}), we see that
\[
\dim_{\rm H}\big\{x \in (0,1): D(x)<1\big\} = \sup_{k \geq 1}\left\{\dim_{\rm H}  \left\{x \in (0,1): D(x)\leq 1 - \frac{1}{k}\right\}\right\}=0.
\]
It follows from \eqref{pq} that $\dim_{\rm H}\{x \in (0,1): D(x) =1\} =0$. Hence $\dim_{\rm H}\{x \in (0,1): D(x) \leq1\} =0$.

For $1<\alpha<\infty$, we have
\[
\big\{x \in (0,1): D(x)\leq \alpha\big\} =  \big\{x \in (0,1): D(x)\leq 1\big\} \bigcup \big\{x \in (0,1): 1<D(x)\leq \alpha\big\}.
\]
Since the first set on the right-hand side is of Hausdorff dimension zero, it remains to estimate the upper bound for the Hausdorff dimension of $\big\{x \in (0,1): 1<D(x)\leq \alpha\big\}$.
To this end, for $k \geq 1$ and $1 \leq j \leq k$, let
\[
E(k,j):=\left\{x \in (0,1): 1+ \frac{j-1}{k}(\alpha -1)<D(x)\leq 1+  \frac{j}{k}(\alpha -1)\right\}.
\]
Then it follows from \eqref{pq} that
\begin{equation}\label{kj}
\dim_{\rm H}E(k,j) \leq  \frac{\frac{\alpha -1}{k}j}{1+ \frac{\alpha -1}{k}(j-1)}.
\end{equation}
For $0<\gamma <1$, we deduce that the map
\[
x \mapsto\frac{\gamma x}{1+\gamma(x-1)}
\]
is increasing in $[1,\infty)$. Hence for $k >\alpha -1$, i.e. $\frac{\alpha-1}{k}<1$,
\[
\max_{1\leq j \leq k}\left\{\frac{\frac{\alpha -1}{k}j}{1+ \frac{\alpha -1}{k}(j-1)}\right\} = \frac{\alpha -1}{1+\frac{k-1}{k}(\alpha -1)}.
\]
Note that for all $k \geq 1$,
\[
\big\{x \in (0,1): 1<D(x)\leq \alpha\big\} = \bigcup^k_{j=1} E(k,j).
\]
Combining these with \eqref{kj}, we finally obtain
\[
\dim_{\rm H} \big\{x \in (0,1): 1<D(x)\leq \alpha\big\} = \max_{1\leq j \leq k}\Big\{\dim_{\rm H} E(k,j)\Big\} \leq \frac{\alpha -1}{1+\frac{k-1}{k}(\alpha -1)}
\]
for any $k >\alpha -1$. Letting $k \to \infty$ yields $\dim_{\rm H} \big\{x \in (0,1): 1<D(x)\leq \alpha\big\} \leq (\alpha -1)/\alpha$. Then the proof is completed.

\end{proof}

We end this paper with the Hausdorff dimension of the set
\[
\Lambda_\phi:=\left\{x\in (0,1): \liminf\limits_{n\to\infty}\frac{\log d_n(x)}{\phi(n)}=1\right\},
\]
where $\phi:\mathbb{N} \rightarrow\mathbb{R}^{+}$ is a function such that $\phi(n) \to \infty$ as $n \to \infty$.
Assume that the limit $\vartheta:=\lim_{n\to\infty} \phi(n)/\log n$ exists.
Theorem \ref{Lambda} implies that if $\vartheta=0$, then the Hausdorff dimension of $\Lambda_\phi$ is zero; if $0<\vartheta<\infty$, then the Hausdorff dimension of $\Lambda_\phi$ is given by the formula in \eqref{SWTheorem3.2}.
For the remaining case, i.e., $\vartheta=\infty$, we have the following theorem which gives a full description of the Hausdorff dimension of $\Lambda_\phi$ for super-logarithmic functions $\phi$.

\begin{theorem}\label{ybphi}
Let $\phi:\mathbb{N} \rightarrow\mathbb{R}^{+}$ be a non-decreasing function and $\phi(n)/\log n \to\infty$ as $n \to \infty$.
Then
\[
\dim_{\rm H}\Lambda_\phi=\frac{1}{A},
 \]
 where $A$ is defined as
 \[
\log A:=\limsup\limits_{n\to\infty}\frac{\log\phi(n)}{n}.
 \]
\end{theorem}

\begin{remark}
Under the same condition as Theorem \ref{ybphi}, Shang and Wu \cite[Theorem 3.1]{SW2} proved that
\[
\dim_{\rm H}\left\{x\in (0,1): \lim\limits_{n\to\infty}\frac{\log d_n(x)}{\phi(n)}=1\right\}=\frac{1}{1+\xi},
\]
where $\xi$ is given by
\[
\xi:= \limsup_{n \to \infty} \frac{\phi(n+1)}{\phi(1)+\cdots +\phi(n)}.
\]
Comparing this with the result of Theorem \ref{ybphi}, we point out that $1+\xi \geq A$ and the strict inequality can be obtained for some special function $\phi$. See Liu \cite{Liu} for the similar result.
\end{remark}

\begin{proof}[Proof of Theorem \ref{ybphi}]
We divide the proof into two parts: the upper bound and the lower bound of $\dim_{\rm H}\Lambda_\phi$.
The upper bound of $\dim_{\rm H}\Lambda_\phi$ is a consequence of Lemma \ref{2} and the lower bound of $\dim_{\rm H}\Lambda_\phi$ relies on Lemma \ref{1} and the arguments in \cite{FS, LM}.

{\bf Upper bound:} For any $0<\varepsilon<1$, we deduce that
\begin{equation*}
\Lambda_\phi \subseteq \bigcup_{N=1}^{\infty}\left\{x\in(0,1): d_n(x)\geq e^{(1-\varepsilon)\phi(n)},\ \forall\,n\geq N\right\}.
\end{equation*}
For $N \in \mathbb{N}$, it follows from Lemma \ref{2} that
\[
\dim_{\rm H}\left\{x\in(0,1): d_n(x)\geq e^{(1-\varepsilon)\phi(n)},\ \forall\,n\geq N\right\} =  \frac{1}{A},
\]
where $A$ is given by $\log A:=\limsup_{n\to\infty} (\log\phi(n))/n$. Then
\[
\dim_{\rm H} \Lambda_\phi \leq \sup_{N\geq 1}\left\{ \dim_{\rm H}\left\{x\in(0,1): d_n(x)\geq e^{(1-\varepsilon)\phi(n)},\ \forall\,n\geq N\right\}\right\}= 1/A.
\]

{\bf Lower bound:}
The number $A$ is given by the equation
\[
\log A=\limsup_{n\to\infty} \frac{\log\phi(n)}{n},
\]
so we have $1\leq A \leq \infty$.
For $A=\infty$, the proof is trivial.
In the following, we always assume that $1 \leq A<\infty$.
For any $\varepsilon>0$, we have $\phi(n) \leq (A+\varepsilon/2)^n$ for $n$ large enough. This implies that for fixed $j \in \mathbb{N}$,
\[
 \phi(n)(A+\varepsilon)^{j-n} \leq (A+\varepsilon/2)^n(A+\varepsilon)^{j-n} \to 0 \ \ \ \text{as} \ \ \ n \to \infty.
\]
Let
\begin{equation*}\label{tidy}
 T_j=\sup\limits_{n\geq j}\left\{e^{\phi(n)(A+\varepsilon)^{j-n}}\right\},\ \forall j=1,2,\cdots.
\end{equation*}
Then the supremum in the definition of $T_j$ is achieved. Since $\phi$ is a non-decreasing function, we have
 \begin{equation}\label{tigx}
 T_{j}\leq T_{j+1}\ \ \ \text{and}\ \ \ \ T_{j+1}\leq T^{A+\varepsilon}_{j}.
 \end{equation}
Here we claim that
\begin{equation}\label{tixjx}
 \liminf\limits_{n\to\infty}\frac{\log T_n}{\phi(n)}=1.
 \end{equation}
In fact, by the definition of $T_j$, we get $T_j \geq e^{\phi(j)}$ for all $j \geq 1$ and then
\[
\liminf\limits_{n\to\infty}\frac{\log T_n}{\phi(n)}\geq1.
\]
For the opposite inequality, denote by $t_j$ the smallest number $k \geq j$ for which $e^{\phi(k)(A+\varepsilon)^{j-k}}$ achieves the supremum in the definition of $T_j$. Namely
\[
t_j:=\inf\left\{k \geq j: e^{\phi(k)(A+\varepsilon)^{j-k}} =T_j\right\}.
\]
Then we obtain $t_j \geq j$ and $t_j \to \infty$ as $j \to \infty$. Next we will show that $t_j=t_{j+1}=\cdots=t_{t_j}$. If $t_j =j$, then the desired result follows.
For $t_j >j$, by the definition of $t_j$, we have
\begin{equation*}
e^{\phi(t_j)(A+\varepsilon)^{j-t_j}} > e^{\phi(\ell)(A+\varepsilon)^{j-\ell}},\ \ \ \forall \ell= j,\cdots,t_j-1
\end{equation*}
and
\begin{equation*}
e^{\phi(t_j)(A+\varepsilon)^{j-t_j}} \geq e^{\phi(m)(A+\varepsilon)^{j-m}},\ \ \ \forall m= t_j+1, t_j+2,\cdots,
\end{equation*}
which yields that
\[
e^{\phi(t_j)(A+\varepsilon)^{j+1-t_j}} > e^{\phi(\ell)(A+\varepsilon)^{j+1-\ell}},\ \ \ \forall \ell= j+1,\cdots,t_j-1
\]
and
\[
e^{\phi(t_j)(A+\varepsilon)^{j+1-t_j}} \geq e^{\phi(m)(A+\varepsilon)^{j+1-m}},\ \ \ \forall m= t_j+1, t_j+2,\cdots
\]
respectively. By the definitions of $T_{j+1}$ and $t_{j+1}$, we see that $t_{j+1} = t_j$. Repeating the above arguments leads to $t_{j+2}=t_{j}$, $t_{j+3}=t_{j}$, $\cdots$, $t_{t_j} = t_{j}$.
Therefore, we obtain $t_j=t_{j+1}=\cdots=t_{t_j}$. Note that $t_j \to \infty$ as $j \to \infty$, so we can choose an increasing subsequence $\{p_k\}$ from $\{t_j\}$.
Then $p_k = t_{p_k}$ and
\[
T_{p_k} = e^{\phi(t_{p_k})(A+\varepsilon)^{p_k-t_{p_k}}}= e^{\phi(p_k)},
\]
which gives
\[
\liminf\limits_{n\to\infty}\frac{\log T_n}{\phi(n)}\leq  \liminf\limits_{n\to\infty}\frac{\log T_{p_k} }{\phi(p_k)}  =1.
\]
Hence \eqref{tixjx} holds.

Since $\{T_j\}$ is non-decreasing and $T_j \to \infty$ as $j \to \infty$, there exists $K>0$ such that $KT_1 \geq 2$.
Write $t_n:=KT_n$ for all $n \geq 1$ and
\[
\mathbb{\widetilde{E}}=\big\{x\in [0,1): nt_n\leq d_n(x)<(n+1)t_n, \forall\ n\geq1\big\}.
\]
Then $\mathbb{\widetilde{E}} \subseteq \Lambda_\phi$.
Note that $\phi(n)/\log n \to\infty$ as $n \to \infty$, combining \eqref{tigx} with \eqref{tixjx}, we have
\[
\limsup\limits_{n\to\infty}\frac{\log(n+1)}{\log T_n} =0\ \ \ \ \text{and}\ \ \ \log T_{n+1}-\log T_1\leq\big(A+\varepsilon-1\big)\sum\limits_{k=1}^{n}\log T_k.
\]
It follows from Lemma \ref{1} that
\[
\dim_{\rm H}\Lambda_\phi \geq \dim_{\rm H}\mathbb{\widetilde{E}} =\frac{1}{1+\eta},
\]
where $\eta$ is given by
\begin{align*}
\eta&= \limsup\limits_{n\to\infty}\frac{\log(n+1)!+\log(KT_{n+1})}{\sum\limits_{k=1}^{n}\log(KT_k)}\\
&\leq \limsup\limits_{n\to\infty}\frac{\log(n+1)!}{\sum\limits_{k=1}^{n} \log T_k}+\limsup_{n\to\infty}\frac{\log T_{n+1}}{\sum\limits_{k=1}^{n}\log  T_k}\\
&\leq  \limsup\limits_{n\to\infty}\frac{\log(n+1)}{\log T_n} +(A+\varepsilon-1)\\
&=A+\varepsilon-1.
\end{align*}
Hence \[\dim_{\rm H}\Lambda_\phi \geq \frac{1}{A+\varepsilon}.\] Letting $\varepsilon \to 0^+$ yields the assertion.
\end{proof}

{\bf Acknowledgement:}
The authors are grateful to Professor Lingmin Liao for his invaluable comments and suggestions.
The research is supported by National Natural Science Foundation of China (No.\,11771153) and Guangdong Natural Science Foundation (No.\,2018B0303110005).
Shang Lei would like to thank China Scholarship Council (No.\,202006150158) for the financial support of her visit to Universit\'{e} Paris-Est Cr\'{e}teil.


\section*{References}


\begin{thebibliography}{10}
\bibitem{Bea} A.F. Beardon, {\it The exponent of convergence of Poincar\'{e} series}, Proc. London Math. Soc. (3) 18 (1968), 461--483.


\bibitem{Edgar} G. Edgar, {\it Integral, Probability, and Fractal Measures}, Springe, New York, 1998.


\bibitem{ERS58} P. Erd\H{o}s, A. R\'{e}nyi and P. Sz\"{u}sz, {\it On Engel's and Sylvester's series}, Ann. Univ. Sci. Budapest. E\"{o}tv\"{o}s. Sect. Math. 1 (1958), 7--32.



\bibitem{Fal90} K. Falconer, {\it Fractal Geometry: Mathematical Foundations and Applications}, John Wiley \& Sons, Ltd., Chichester, 1990.




\bibitem{FS} L. Fang and K. Song, {\it Multifractal analysis of the convergence exponent in continued fractions}, arXiv:1911.01821.

\bibitem{FWnon} L. Fang and M. Wu, {\it Hausdorff dimension of certain sets arising in Engel expansions}, Nonlinearity 31 (2018), 2105--2125.






\bibitem{lesGal76} J. Galambos, {\it Representations of Real Numbers by Infinite Series}, Lecture Notes in Mathematics, Vol. 502. Springer-Verlag, Berlin-New York, 1976.




%





\bibitem{KS82} P. Kostyrko and T. \u{S}al\'{a}t, {\it On the exponent of convergence}, Rend. Circ. Mat. Palermo (2) 31 (1982), 187--194.





\bibitem{LM} L. Liao and M. Rams, {\it Upper and lower fast Khintchine spectra in continued fractions}, Monatsh. Math. 180 (2016), 65--81.








\bibitem{Liu} J. Liu, {\it On some exceptional sets in Engel expansions and Hausdorff dimensions}, Fractals, 28 (2020), 2050140.









\bibitem{LW03} Y. Liu and J. Wu, {\it Some exceptional sets in Engel expansions}, Nonlinearity 16 (2003), 559--566.
%







%

\bibitem{Oxt} J. Oxtoby, {\it Measure and Category. A Survey of the Analogies Between Topological and Measure Spaces}, Graduate Texts in Math., Vol. 2, Springer-Verlag, Berlin and New York, 1980.





\bibitem{PS} G. P\'{o}lya and G. Szeg\H{o}, {\it Problems and Theorems in Analysis I}, Springer, 1972.









\bibitem{Salat} T. \u{S}al\'{a}t, {\it On exponents of convergence of subsequences}, Czechoslovak Math. J. 34 (1984), 362--370.





\bibitem{SW} L. Shang and M. Wu, {\it Slow growth rate of the digits in Engel expansions}, Fractals, 28 (2020), 2050047.

\bibitem{SW2} L. Shang and M. Wu, {\it On the growth speed of digits in Engel expansions}, J. Number Theory 219 (2021), 368--385.




\end{thebibliography}
\end{document}